\documentclass[12pt,a4paper]{article}
\usepackage{bm}
\usepackage{url}
\usepackage{chngcntr}
\usepackage{amsmath}
\usepackage{amssymb}
\usepackage{amsthm}
\usepackage{enumitem}
\usepackage{setspace}
\usepackage{graphicx}
\usepackage{todonotes}
\usepackage[nottoc,notlot,notlof]{tocbibind}
\usepackage{authblk}
\usepackage[top=2cm, bottom=2cm, left=2cm, right=2cm]{geometry}
\usepackage{mathtools}
\usepackage{dirtytalk}
\usepackage[autostyle]{csquotes}
\input xy
\xyoption{all}

\newtheorem{theorem}{Theorem}[section]
\newtheorem{lemma}[theorem]{Lemma}

\theoremstyle{definition}

\newtheorem*{theorem*}{Theorem}
\newtheorem{definition}[theorem]{Definition}

\newtheorem{remark}[theorem]{Remark}

\newcommand{\R}{\mathbb{R}}

\newcommand{\F}{\mathbb{F}}
\newcommand{\N}{\mathbb{N}}

\newcommand{\PP}{\mathbb{P}}

\newcommand{\hr}{_\textnormal{hr}}
\newcommand{\st}{_\textnormal{st}}

\newcommand{\lr}{_\textnormal{lr}}

\newcommand{\E}[2][]{\mathop{\mathbb{E}}_{#1}\left[ #2 \right]}

\usepackage{xcolor}

\newcommand{\beq}{\begin{eqnarray*}}
\newcommand{\eeq}{\end{eqnarray*}}
\newcommand{\beqn}{\begin{eqnarray}}
\newcommand{\eeqn}{\end{eqnarray}}

\counterwithout{equation}{section}

\begin{document}
\pagenumbering{arabic}
\setcounter{page}{1}

\title{The restrictiveness of the hazard rate order and the moments of the maximal coordinate of a random vector uniformly distributed on the probability $n$-simplex}

\author{Sela Fried \thanks{The author is a postdoctoral fellow in the Department of Computer Science at the Ben-Gurion University of the Negev. Research Supported by the Israel Science Foundation (ISF) through grant No.
1456/18 and European Research Council Grant number: 949707.}}
\date{} 
\maketitle

\begin{abstract}
Continuing the work of $\cite{Fried2021OnTR}$ who defined the restrictiveness of stochastic orders and calculated the restrictiveness of the usual stochastic order and the likelihood ratio order, we calculate the restrictiveness of the hazard rate order. Inspired by the works of \cite{onn2011generating} and \cite{weissman2011testing}, we propose a possible application of the restrictiveness results in randomness testing. We then apply a dimension reduction technique, that proved useful in obtaining the restrictiveness results, and provide an alternative proof for Whitworth's formula. By integrating the formula, we derive the moments of the maximal coordinate of a random vector that is uniformly distributed on the probability $n$-simplex.
\end{abstract} 

\section{Introduction}

This work studies two aspects of random vectors that are uniformly distributed on the probability $n$-simplex $\Delta^n$. First, we continue the project of quantifying the restrictiveness of stochastic orders, initiated recently by \cite{Fried2021OnTR} who defined the restrictiveness of stochastic orders and calculated the restrictiveness of the usual stochastic order and the likelihood ratio order. Stochastic orders are partial orders on probability distributions and are not, in general, total orders, i.e., two arbitrary probability distributions are not necessarily comparable with respect to a certain stochastic order. In particular, (the volume of) the set of all probability distributions that are comparable to a fixed probability distribution $P$ is a function of $P$. One possibility to obtain a measure of restrictiveness of stochastic orders that is independent of a particular choice of $P$ would be to average over all $P$. This path was taken by  \cite{Fried2021OnTR} who defined the restrictiveness of a stochastic order as the probability that two randomly and uniformly chosen probability distributions are comparable with respect to the stochastic order. Here, we calculate the restrictiveness of the hazard rate order, a common and important stochastic order that finds applications in reliability theory and survival analysis. The reader is referred to  \cite{shaked2007stochastic} for a comprehensive study of stochastic orders and to \cite{boland1994applications} for a review of works on the hazard rate order.

The main insight that is at the heart of all the results on the restrictiveness of stochastic orders is that, conditioned on one of the coordinates, comparability under a certain stochastic order may be verified in one dimension less. We apply this dimension reduction technique and provide an alternative solution to an old problem that was originally formulated as follows:

\bigskip

\textbf{The line version} A line of length $1$ is divided into $n\geq 2$ segments by $n-1$ random points. What is the probability that the length of the longest segment is less than or equal to $b$ for $b>0$? 

\bigskip

This problem whose solution is referred to in \cite[Exercise 5 on p. 213]{devroye2006nonuniform} as Whitworth's formula due to its seemingly first appearance in \cite[667 on p. 196]{whitworth1897dcc}, has numerous applications in statistics. It was formulated and its solution rediscovered many times (\cite[p. 252]{darling1953class}, \cite{whitworth1897dcc}, \cite{fisher1929tests}, \cite{garwood1940application}. The latter gives more references. Recently, \cite{pinelis2019order} treated the problem in a much broader scale). By randomly and uniformly drawing $n-1$ independent points in $[0,1]$, sorting them and taking differences (cf. \cite[Algorithm 2.5.3]{rubinstein2016simulation} or \cite[Theorem 2.1]{devroye2006nonuniform}), it is clear that the above problem may be formulated as follows:

\bigskip

\textbf{The simplex version} What is the probability that the maximal coordinate of a uniformly and randomly chosen point from $\Delta^n$ is less than or equal to $b$ for $b>0$?

\bigskip

We address the problem in its simplex version and upon establishing Whitworth's formula, we derive an expression for the moments of the maximal coordinate of a random vector uniformly distributed on $\Delta^n$, which, to the best of our knowledge, have not been calculated beyond the first two. These formulas involve sums of harmonic numbers and of generalized harmonic numbers, providing a setting where these appear naturally (cf. \cite{choi2011some}).

\bigskip

We conclude the work by proposing a possible application of the restrictiveness results in randomness testing that was inspired by the works of \cite{onn2011generating} and \cite{weissman2011testing}. 

\bigskip

The work is structured as follows: In the next section we go over our main results. Then, after a short preliminary section, Section \ref{sec; hr} addresses the hazard rate order and Section \ref{sec; max} the maximal coordinate. In section \ref{sec; app} we show how the restrictiveness of stochastic orders might be applied in randomness testing. %Finally, in Section \ref{sec; dis} we discuss 

\section{Main results}

Let $n\in\N$ and let $\theta=(\theta_1,\ldots,\theta_n)\in\R^n$ such that $\theta_1 +\cdots+\theta_n=1$ and $\theta_1,\ldots,\theta_n\geq 0$. Thus, $\theta$ is an element of the probability $n$-simplex $\Delta^n$. We study random vectors $\Theta$ that are uniformly distributed on $\Delta^n$, written $\Theta\sim\mathcal{U}\left(\Delta^n\right)$. In order to make the meaning of $\Theta\sim\mathcal{U}\left(\Delta^n\right)$ precise, we give the joint cumulative distribution function of such vectors. It is easily derived from known results on intersections between hyperplanes and hypercubes, but we could not find it written explicitely:

\begin{lemma}
Let $n\in\N$ and  $(\theta_1,\ldots,\theta_n)\in\R^n$. Suppose $\Theta=(\Theta_1,\ldots,\Theta_n)\sim\mathcal{U}\left(\Delta^n\right)$. Then $$\mathbb{P}(\Theta_1\leq\theta_1,\ldots,\Theta_n\leq\theta_n)= \begin{cases}
    \sum_{K\subseteq[n]}(-1)^{|K|}\left(\max\left\{1-\sum_{k\in K}\theta_k,0\right\}\right)^{n-1}&\textnormal{if }\theta_1,\ldots,\theta_n>0\\0&\textnormal{otherwise}. \end{cases}$$ 
\end{lemma}

Coming to the restrictiveness results, recall that every $\theta\in\Delta^n$ induces a probability distribution $\PP_\theta$ of a random variable $X$ that can assume at most $n$ real numbers $x_1<\cdots<x_n$ by defining $\PP_\theta(X = x_i) = \theta_i, \;1\leq i\leq n$. We identify $\PP_\theta$ with $\theta$. Let $\theta=(\theta_1,\ldots,\theta_n),\theta'=(\theta'_1,\ldots,\theta_n')\in\Delta^n$. Recall (e.g. \cite[1.B.10]{shaked2007stochastic}) that $\theta$ is said to be smaller than $\theta'$ in the hazard rate order, written $\theta\leq\hr\theta'$, if for every $1\leq i\leq j\leq n$ it holds: $$\Big(\sum_{k=i}^n \theta_k\Big)\Big(\sum_{k=j}^n \theta'_k\Big)\geq\Big(\sum_{k=j}^n \theta_k\Big)\Big(\sum_{k=i}^n \theta'_k\Big).$$

We prove

\begin{theorem} 
Let $n\in\mathbb{N}$ and suppose $\Theta,\Theta'\sim\mathcal{U}\left(\Delta^n\right)$ are independent. Then
\begin{equation}\label{eq; 987} 
\PP(\Theta\leq\hr \Theta')=\frac{1}{2^{n-1}}.
\end{equation}
\end{theorem} 

It is interesting to compare (\ref{eq; 987}) with the analogue results of \cite[Theorem on p. 1]{Fried2021OnTR}: $$\PP(\Theta\leq\st \Theta')=\frac{1}{n}\;\; \textnormal{ and }\;\; \PP(\Theta\leq\lr \Theta')=\frac{1}{n!}$$ where $\leq\st$ and $\leq\lr$ denote the usual stochastic order and the likelihood ratio order, respectively. It follows that the likelihood ratio order is more restrictive than the hazard rate order, which, in turn, is more restrictive than the usual stochastic order. This is in agreement with the well known relationship between these stochastic orders (e.g. \cite[Theorem 1.C.1. and Theorem 1.B.1.]{shaked2007stochastic}): Let $\theta, \theta'\in\Delta^{n}$. Then $$\theta\leq\lr \theta'\Longrightarrow \theta\leq\hr \theta'\Longrightarrow  \theta\leq\st \theta'.$$ 

We now come to the part regarding the maximal coordinate $\max(\Theta) = \max_{1\leq i\leq n}\{\Theta_i\}$ of $\Theta=(\Theta_1,\ldots,\Theta_n)\sim\mathcal{U}\left(\Delta^n\right)$. By applying a dimension reduction technique we obtain an alternative proof for Whitworth's formula:

\begin{theorem}\label{thm; 1313}
Let $2\leq n\in\N$ and $\frac{1}{n}< b\leq 1$. Suppose $\Theta=(\Theta_1,\ldots,\Theta_n)\sim\mathcal{U}\left(\Delta^n\right)$.  Then \begin{equation}\label{eq; 555}
\PP\left(\max(\Theta)\leq b\right) = \sum_{k=0}^{m}\binom{n}{k}(-1)^k((n-k)b-1)^{n-1}
\end{equation} where $0\leq m\leq n-2$ is such that $\frac{1}{n-m}<b\leq \frac{1}{n-m-1}$.
\end{theorem}

Before we present our result regarding the moments of $\max(\Theta)$, let us review what has already been done in this respect: It was shown by \cite{onn2011generating} in two different ways (geometric and probabilistic) that \begin{equation}\label{eq; ex} \E{\max(\Theta)}=\frac{1}{n}\sum_{k=1}^n\frac{1}{k}.\end{equation} Their probabilistic proof relies on the fact that if $Y_1,\ldots,Y_n$ are iid exponential random variables and $Z=\sum_{i=1}^n Y_i$ then $\left(\frac{Y_1}{Z},\ldots,\frac{Y_n}{Z}\right)\sim\mathcal{U}(\Delta^n)$ (cf. \cite[Theorem 2.2]{devroye2006nonuniform}). They also remarked (\cite[Remark 3]{onn2011generating}) that their arguments allow to compute other moments and demonstrated this by calculating the second moment. But a) their expression for the second moment contains a typo and b) they provided neither the moments of $\max_{1\leq i\leq n}\{Y_i\}$ that are crucial in their approach nor a reference to them. A more direct approach for calculating the moments would be to integrate (\ref{eq; 555}) over $[0,1]$. Actually, this was proposed by them (\cite[Remark 2]{onn2011generating}) as an additional way to derive (\ref{eq; ex}). To the best of our knowledge, the following formula for the moments of $\max(\Theta)$ has not appeared before:

\begin{theorem}
Let $\Theta\sim\mathcal{U}(\Delta^n)$ and let $t\in\N$. Then $$\E{\max(\Theta)^t}=\frac{1}{\binom{n-1+t}{t}}\sum_{k=1}^{n}(-1)^{k-1}\binom{n}{k}\frac{1}{k^t}.$$ In particular, $$\E{\max(\Theta)} = \frac{1}{n}\sum_{k=1}^{n}\frac{1}{k}\;\;\;\textnormal{ and }\;\;\;\textnormal{Var}(\max(\Theta))=\frac{1}{n^{2}(n+1)}\left(n\sum_{k=1}^{n}\frac{1}{k^{2}}-\left(\sum_{k=1}^{n}\frac{1}{k}\right)^{2}\right).$$
\end{theorem}

\section{Preliminaries}\label{sec; pre}

Henceforth, $u$ is a positive real number and unless stated otherwise, $n\in\N$ satisfies $n\geq 2$. We denote $[n]=\{1,2,\ldots,n\}$ and $\mathbf{1}=(1,1,\ldots,1)\in\R^n$.

\begin{definition}\label{def; simplex}
The \emph{$n$-simplex (of size $u$)} is defined to be
$$ \Delta^{n,u} =   \{(\theta_1,\ldots,\theta_n)\in\R^n\; |\; 
  \theta_1 + \cdots + \theta_n = u,\;  \theta_i \geq 0, \;1\leq i \leq n\}.$$
If $u=1$ then $\Delta^n = \Delta^{n,1}$ is merely the \emph{probability $n$-simplex}.
\end{definition}

\begin{remark}
The reason for considering arbitrary $n$-simplices (as opposed to restricting to the probability $n$-simplex) lies in the inductive approach we take which results in vectors whose coordinates do not necessarily sum to $1$.
\end{remark}

%(cf. \cite{wheeden1977measure}). 

This work studies random vectors that are uniformly distributed on $\Delta^{n,u}$. Such vectors will be denoted by $\Theta^{n,u} = (\Theta_1^{n,u},\ldots,\Theta_n^{n,u})$ and we will write $\Theta^{n,u}\sim\mathcal{U}(\Delta^{n,u})$. If confusion is unlikely, we write shortly $\Theta = \Theta^{n,u}$ and $\Theta_i = \Theta_i^{n,u}, 1\leq i\leq n$. We shall use lowercase letter such as $\theta=(\theta_1,\ldots,\theta_n), x=(x_1,\ldots,x_n)$, etc., to denote points in $\R^n$. For $\Theta\sim\mathcal{U}(\Delta^{n,u})$ the random variable $\max(\Theta)$ is defined by $$
\max(\Theta)=\max_{1\leq i\leq n}\{\Theta_i\}.$$

We wish to derive the joint cumulative distribution function of $\Theta\sim\mathcal{U}(\Delta^{n,u})$. Notice that ‘uniformly’ means that for a  subset $C\subseteq\Delta^{n,u}$ it holds $$ \PP(\Theta\in C)=\frac{\text{Vol}(C)}{\text{Vol}(\Delta^{n,u})}=\frac{\int_{C} 1dV}{\int_{\Delta^{n,u}} 1dV}$$ where both integrals are over $n-1$-manifolds (cf. \cite[\S 25]{munkres2018analysis}). Thus, we shall make extensive use of the volume of the $n$-simplex:

\begin{lemma}\label{lem 2}
It holds $$\textnormal{Vol}(\Delta^{n,u}) = \frac{\sqrt{n}u^{n-1}}{(n-1)!}.$$ 
\end{lemma}
\begin{proof}
See, for example, \cite[Lemma 2.2]{Fried2021OnTR}.
\end{proof}

\begin{lemma}\label{lem; uni}
Suppose $\Theta\sim\mathcal{U}(\Delta^{n,u})$ and let  $(\theta_1,\ldots,\theta_n)\in\R^n$. The joint cumulative distribution function of $\Theta$ is given by $$\mathbb{P}(\Theta_1\leq\theta_1,\ldots,\Theta_n\leq\theta_n)=\frac{1}{u^{n-1}} \begin{cases}
    \sum_{K\subseteq[n]}(-1)^{|K|}\left(\max\left\{u-\sum_{k\in K}\theta_k,0\right\}\right)^{n-1}&\textnormal{if }\theta_1,\ldots,\theta_n>0\\0&\textnormal{otherwise}. \end{cases}$$
\end{lemma}

\begin{proof}
The claim follows from \cite[Theorems 1 and 4]{marichal2008slices} (and their proofs) with some minor adjustments. Here, we only sketch the proof and the reader is referred to \cite{marichal2008slices} for the details. 
For $w=(w_1,\ldots,w_n)\in\R^n$ we denote \begin{align}G^{w,u}=&\{(x_1,\ldots,x_n)\in\R^n\;|\; \sum_{i=1}^n w_i x_i\leq u\},\nonumber \\ I^w =&\{(x_1,\ldots,x_n)\in\R^n\;|\; 0 \leq x_i\leq w_i,\; 1\leq i\leq n\}\nonumber. \end{align}
Set $\theta=(\theta_1,\ldots,\theta_n)$.
The same reasoning in the proof of \cite[Theorem 4]{marichal2008slices} may be applied in order to show that
$$
\text{Vol}(\Delta^{n,u}\cap I^\theta)=\sqrt{n}\frac{\partial}{\partial u}\text{Vol}(G^{\mathbf{1},u}\cap I^\theta).$$ The substitution $(x_1,\ldots,x_n)\mapsto(\theta_1x_1,\ldots,\theta_nx_n)$ gives $$\text{Vol}(G^{\mathbf{1},u}\cap I^\theta)=\left(\prod_{i=1}^n\theta_i\right)\text{Vol}(G^{\theta,u}\cap I^\mathbf{1}).$$ By \cite[Theorem 1]{marichal2008slices}, $$\text{Vol}(G^{\theta,u}\cap I^\mathbf{1})=\frac{1}{n!\left(\prod_{i=1}^n\theta_i\right)}\sum_{K\subseteq[n]}(-1)^{|K|}\left(\max\left\{u-\sum_{k\in K}\theta_k,0\right\}\right)^n.$$ Thus, $$\text{Vol}(\Delta^{n,u}\cap I^\theta)=\frac{\sqrt{n}}{(n-1)!}\sum_{K\subseteq[n]}(-1)^{|K|}\left(\max\left\{u-\sum_{k\in K}\theta_k,0\right\}\right)^{n-1}.$$ Dividing both sides of the equation by $\text{Vol}(\Delta^{n,u})$ completes the proof.
\end{proof}

\begin{remark}
 In the notation of the previous lemma, \cite[Lemma 2.1]{devroye2006nonuniform} shows that for $\theta_1,\ldots,\theta_n\geq 0$ it holds $$P(\Theta_1>\theta_1,\ldots,\Theta_n>\theta_n) =\left(\max\left\{1-\sum_{k=1}^n \theta_k,0\right\}\right)^{n-1}.$$ 
\end{remark}

\iffalse\begin{definition}\label{def; par} 
Let $\preceq$ be a partial order on $\Delta^{n,u}$ and $ \theta\in\Delta^{n,u}$. We denote $$\Delta^{n,u}_{\succeq \theta} = \{\theta'\in\Delta^{n,u}\;|\;\theta'\succeq \theta\} .$$
\end{definition}
\fi

\section{The hazard rate order}\label{sec; hr}

\bigskip
The following definition is a modification of the definition in \cite[1.B.10 on p. 17]{shaked2007stochastic}:

\begin{definition}\label{def; fosd}
Let $\theta=(\theta_1,\ldots,\theta_n), \theta'=(\theta'_1,\ldots,\theta'_n)\in \Delta^{n,u}$. We say that \emph{$\theta$ is smaller than $\theta'$ in the hazard rate order} and write $\theta\leq\hr \theta'$ if \begin{equation}\label{eq; 67} \Big(\sum_{k=i}^n \theta_k\Big)\Big(\sum_{k=j}^n \theta'_k\Big)\geq\Big(\sum_{k=j}^n \theta_k\Big)\Big(\sum_{k=i}^n \theta'_k\Big), \;\;\forall 1\leq i\leq j\leq n.\end{equation}
\end{definition}

The following lemma shows that given the last coordinate, comparability with respect to the hazard rate order can be verified in one dimension less. Its proof is easy and we omit it.

\begin{lemma}\label{lem 1}
Let $\theta=(\theta_1,\ldots,\theta_{n+1}), \theta'=(\theta'_1,\ldots,\theta'_{n+1})\in\Delta^{n+1,u}$ and assume $\theta'_{1}<u$. Then $\theta\leq\hr \theta'  \;(\text{in } \Delta^{n+1,u})$ if and only if $\theta'_1\leq\theta_1$ and
$$
(\theta_2/v, \ldots, \theta_{n+1}/v) \leq\hr  (\theta'_2,\ldots,\theta'_{n+1})\;( \text{in }\Delta^{n,u-\theta'_1}) \text{ where } v = \frac{\sum_{k=2}^{n+1}\theta_k}{u-\theta'_1}.
$$
\end{lemma}

\begin{lemma}\label{lem 3}
Suppose $\Theta\sim\mathcal{U}(\Delta^{n,u})$ and let $\theta=(\theta_1,\ldots,\theta_n)\in\Delta^{n,u}$. Then 
$$P(\Theta\geq\hr \theta)=\prod_{i=1}^{n-1}\frac{\Big(\sum_{j=i}^n \theta_j\Big)^{n-i}-\Big(\sum_{j=i+1}^n \theta_j\Big)^{n-i}}   {\Big(\sum_{j=i}^n \theta_j\Big)^{n-i}} .$$
\end{lemma}

\begin{proof}
We proceed by induction. 
For $n=2$, condition (\ref{eq; 67}) comes down to $\theta'_2\geq\theta_2$. Thus, $$P(\Theta\geq\hr \theta)= \frac{1}{u}\int_{\theta_2}^ud\theta_2'=\frac{u-\theta_2}{u}.$$
Suppose the claim holds for $n$ and let $\theta=(\theta_1,\ldots,\theta_{n+1})\in\Delta^{n+1,u}$. Then 
\begin{align}
P(\Theta^{n+1,u}\geq\hr \theta) = &\frac{n!}{u^n}\int_0^{\theta_1}\frac{(u-\theta'_1)^{n-1}}{(n-1)!}P(\Theta^{n,u-\theta'_1}\geq\hr(\theta_2/v,\ldots,\theta_{n+1}/v))d\theta'_1 \nonumber \\ =&\frac{1}{u^n}\prod_{i=1}^{n-1}\frac{\Big(\sum_{j=i+1}^{n+1} \theta_j\Big)^{n-i}-\Big(\sum_{j=i+2}^{n+1} \theta_j\Big)^{n-i}} {\Big(\sum_{j=i+1}^{n+1} \theta_j\Big)^{n-i}}(u^n-(u-\theta_1)^n) \nonumber \\ = & \prod_{i=1}^n\frac{\Big(\sum_{j=i}^{n+1} \theta_j\Big)^{n+1-i}-\Big(\sum_{j=i+1}^{n+1} \theta_j\Big)^{n+1-i}} {\Big(\sum_{j=i}^{n+1} \theta_j\Big)^{n+1-i}}\nonumber.
\end{align}
\end{proof}
We come now to the main result of this section:
\begin{theorem}
Suppose $\Theta,\Theta'\sim\mathcal{U}(\Delta^{n,u})$ are independent. Then $$P(\Theta \leq\hr \Theta')=\frac{1}{2^{n-1}}.$$
\end{theorem}
\begin{proof}
It holds 
\begin{align}
P(\Theta \leq\hr \Theta') = & \frac{(n-1)!}{\sqrt{n}u^{n-1}}\int_{\Delta^{n,u}} \prod_{i=1}^{n-1}\frac{\left(\sum_{j=i}^n \theta_j\right)^{n-i}-\left(\sum_{j=i+1}^n \theta_j\right)^{n-i}}   {\left(\sum_{j=i}^n \theta_j\right)^{n-i}}dV\nonumber \\ =& \frac{(n-1)!}{u^{n-1}}\int_0^u\int_0^{u-\theta_2} \cdots\int_0^{u-\sum_{i=2}^{n-1}\theta_i} \prod_{i=1}^{n-1}\frac{\left(\sum_{j=i}^n \theta_j\right)^{n-i}-\left(\sum_{j=i+1}^n \theta_j\right)^{n-i}}   {\left(\sum_{j=i}^n \theta_j\right)^{n-i}} d\theta_n\cdots d\theta_2 \nonumber  \\
=& \frac{(n-1)!}{u^{n-1}}\int_0^u\int_0^{u-\theta_2} \cdots\int_0^{u-\sum_{i=2}^{n-1}\theta_i}\frac{\left(\sum_{j=1}^n\theta_j\right)^{n-1}-\left(\sum_{j=2}^n\theta_j\right)^{n-1}}{\left(\sum_{j=1}^n\theta_j\right)^{n-1}} \cdot \nonumber \\ & \hspace{5.5cm}\prod_{i=2}^{n-1}\frac{\left(\sum_{j=i}^n \theta_j\right)^{n-i}-\left(\sum_{j=i+1}^n \theta_j\right)^{n-i}}   {\left(\sum_{j=i}^n \theta_j\right)^{n-i}} d\theta_n\cdots d\theta_2 \nonumber  \\ =& \frac{(n-1)!}{u^{n-1}}\int_0^u\int_0^{u-\theta_2} \cdots\int_0^{u-\sum_{i=2}^{n-1}\theta_i}\prod_{i=2}^{n-1}\left(1-\frac{\left(\sum_{j=i+1}^n \theta_j\right)^{n-i}}{\left(\sum_{j=i}^n \theta_j\right)^{n-i}}\right) d\theta_n\cdots d\theta_2 -\nonumber \\ &\hspace{1cm} \frac{(n-1)!}{u^{2(n-1)}}\int_0^u\int_0^{u-\theta_2} \cdots\int_0^{u-\sum_{i=2}^{n-1}\theta_i}\left(\left(\sum_{j=2}^n \theta_j\right)^{n-1}-\left(\sum_{j=2}^n \theta_j\right)\left(\sum_{j=3}^n \theta_j\right)^{n-2}\right)\cdot \nonumber \\& \hspace{5.5cm} \prod_{i=3}^{n-1}\left(1-\frac{\left(\sum_{j=i+1}^n \theta_j\right)^{n-i}}{\left(\sum_{j=i}^n \theta_j\right)^{n-i}}\right)d\theta_n\cdots d\theta_2.  \label{aa}
\end{align}
Consider the following substitution which is a variation of \cite[Exercise 9.13.1]{shurman2016calculus}: $$\theta_i=
\begin{cases}
    (1-y_{i+1}) \prod_{j=2}^iy_j,& 2\leq i \leq n-1\\
    \prod_{j=1}^ny_j,              & i=n.
\end{cases}$$ It is easily verified that the Jacobian is given by $\prod_{i=2}^{n-1} y_i^{n-i}$ and that $$\prod_{j=2}^iy_j=\sum_{j=i}^n \theta_j, \;\;2\leq i\leq n.$$ Thus, 

\begin{align}
(\ref{aa})  = &\frac{(n-1)!}{u^{n-1}}\int_0^u y_2^{n-2}dy_2\prod_{i=3}^{n-1}\int_0^1(1-y_{i}^{n+1-i})y_i^{n-i}dy_i\int_0^1 1-y_ndy_n -\nonumber \\ &\hspace{10 pt} \frac{(n-1)!}{u^{2(n-1)}}\int_0^u\int_0^1y_2^{2n-3}y_3^{n-3}-y_2^{2n-3}y_3^{2n-5}dy_3dy_2 \prod_{i=4}^{n-1}\int_0^1(1-y_{i}^{n+1-i})y_i^{n-i} dy_i\int_0^11-y_ndy_n \nonumber\\= &\frac{(n-1)!}{u^{n-1}}\frac{u^{n-1}}{n-1}\prod_{i=3}^{n-1}\frac{1}{2(n+1-i)}\frac{1}{2} -\nonumber \\ &\hspace{10 pt} \frac{(n-1)!}{u^{2(n-1)}}\left(\frac{u^{2(n-1)}}{2(n-1)(n-2)}-\frac{u^{2(n-1)}}{2(n-2)2(n-1)}\right)\prod_{i=4}^{n-1}\frac{1}{2(n+1-i)}\frac{1}{2} \nonumber \\ =& \frac{1}{2^{n-2}}- \frac{1}{2^{n-1}}=\frac{1}{2^{n-1}}. \nonumber
\end{align}
\end{proof}

\section{$\max(\Theta)$}\label{sec; max}

Identity (\ref{eq; 7}) in the following lemma is referred to by \cite[Exercise 5 on p. 213]{devroye2006nonuniform} as Whitworth's formula and goes back at least to 1897 when it appeared as an exercise in \cite[667 on p. 196]{whitworth1897dcc}. It has many proofs but, to the best of our knowledge, a proof that is based on a dimension reduction (inductive) has not appeared before. In Theorem \ref{thm; 5.7} we will integrate the right-hand side of (\ref{eq; 7}) over $[0,u]$ in order to derive the moments of $\max(\Theta)$.

\begin{lemma}
Suppose $\Theta\sim\mathcal{U}(\Delta^{n,u})$ and let $\frac{u}{n}< b\leq u$. Then \begin{equation}\label{eq; 7}\PP(\max(\Theta)\leq b) = \frac{1}{u^{n-1}}\sum_{k=0}^{m}\binom{n}{k}(-1)^k((n-k)b-u)^{n-1}\end{equation} where $0\leq m\leq n-2$ is such that $\frac{u}{n-m}<b\leq \frac{u}{n-m-1}$.
\end{lemma}

\begin{proof}
We proceed by induction. For $n=2$, necessarily, $m=0$. Then $$\PP(\max(\Theta)\leq b) = \frac{1}{u}\int_{u-b}^bd\theta_1 = \frac{2b-u}{u}.$$ Assume now that the claim holds for $n$ and let $0\leq m\leq n-1$ such that $\frac{u}{n+1-m}<b\leq\frac{u}{n+1-m-1}$. Before we proceed, let us observe the following: Let $(\theta_1,\ldots,\theta_{n+1})\in\Delta^{n+1,u}$ such that $\theta_1\leq b$. Then $$\max_{1\leq i\leq n+1}\{\theta_i\}\leq b\Longrightarrow \theta_1 \geq u-bn.$$ We distinguish between two cases: Suppose $m=0$, i.e., $\frac{u}{n+1}<b\leq\frac{u}{n}$. Equivalently, $0\leq u-nb<b$. Now, let $u-nb< \theta_1\leq b$. Then $\frac{u-\theta_1}{n}<b$. Additionally, $b\leq \frac{u-\theta_1}{n-1}\iff \theta_1\leq u-(n-1)b$ and we have that $b\leq u-(n-1)b$ since $nb\leq u$. We have just shown that $$u-nb<\theta_1\leq b\Longrightarrow \frac{u-\theta_1}{n}<b\leq\frac{u-\theta_1}{n-1}.$$
Thus, \begin{align}
\PP(\max(\Theta^{n+1,u})\leq b) = & \frac{n}{u^n}\int_{u-nb}^b (u-\theta_1)^{n-1}\PP(\max(\Theta^{n,u-\theta_1})\leq b) d\theta_1\nonumber\\=& \frac{n}{u^n}\int_{u-nb}^b (nb-u+\theta_1)^{n-1} d\theta_1\nonumber\\ = &\frac{1}{u^{n}}((n+1)b-u)^{n}\nonumber.
\end{align}
Suppose now $1\leq m\leq n-2$. As in the previous case, it can be shown that $$u-nb\leq u-(n+1-m)b<0\leq u-(n-m)b<b,$$ $$0\leq\theta_1\leq u-(n-m)b\Longrightarrow\frac{u-\theta_{1}}{n-(m-1)}<b\leq\frac{u-\theta_{1}}{n-(m-1)-1}\text{ and}$$ $$u-(n-m)b<\theta_1\leq b \Longrightarrow\frac{u-\theta_{1}}{n-m}<b\leq\frac{u-\theta_{1}}{n-m-1}.$$ Thus,
\begin{align}
\PP(\max(\Theta^{n+1,u})\leq b) =&  \frac{n}{u^n}\Bigg(\int_{0}^{u-(n-m)b} (u-\theta_1)^{n-1}\PP(\max(\Theta^{n,u-\theta_1})\leq b) d\theta_1+\nonumber\\ &\hspace{100pt}\int_{u-(n-m)b}^b (u-\theta_1)^{n-1}\PP(\max(\Theta^{n,u-\theta_1})\leq b) d\theta_1\Bigg)\nonumber\\ =& \frac{n}{u^n}\Bigg(\int_{0}^{u-(n-m)b} \sum_{k=0}^{m-1}\binom{n}{k}(-1)^k((n-k)b-u+\theta_1)^{n-1} d\theta_1+\nonumber\\ &\hspace{100pt} \int_{u-(n-m)b}^b \sum_{k=0}^{m}\binom{n}{k}(-1)^k((n-k)b-u+\theta_1)^{n-1}  d\theta_1\Bigg)\nonumber\\ =& \frac{1}{u^n}\Bigg( \sum_{k=0}^{m-1}\binom{n}{k}(-1)^k\left(((m-k)b)^n-((n-k)b-u)^n\right) +\nonumber\\ &\hspace{100pt}  \sum_{k=0}^{m}\binom{n}{k}(-1)^k\left(((n+1-k)b-u)^n-((m-k)b)^n\right) \Bigg)\nonumber\\ =& \frac{1}{u^n}\Bigg(((n+1)b-u)^n+ \sum_{k=1}^{m}\left(\binom{n}{k}+\binom{n}{k-1}\right)(-1)^k((n+1-k)b-u)^n \Bigg)\nonumber\\ =& \frac{1}{u^n} \sum_{k=0}^{m}\binom{n+1}{k}(-1)^k((n+1-k)b-u)^n. \nonumber
\end{align}

Finally, assume $m=n-1$. This case is similar to the previous one with the exception that for $u-b\leq \theta_1\leq b$ it holds $u-\theta_1\leq b$ and therefore $\PP(\max(\Theta^{n,u-\theta_1})\leq b)=1$. We leave out the details.
\end{proof}

Before we derive the moments of $\max(\Theta)$ several preparations are necessary. The first is the a special case of a difference between incomplete beta functions (cf. \cite[6.6.1]{abramowitz+stegun}) that may be proved easily by induction:

\begin{lemma}
Let $a,b>0$. Then for every $p,q\in\N\cup\{0\}$ it holds $$\int_{\frac{u}{b}}^{\frac{u}{a}}x^{p}(bx-u)^{q}dx=\left(\frac{u}{a}\right)^{p+q+1}\sum_{k=1}^{p+1}(-1)^{k-1}\frac{(b-a)^{q+k}\prod_{j=1}^{k-1}(p-j+1)}{b^k\prod_{j=1}^k(q+j)}.$$
\end{lemma}

Next, we define a function that appears in the formula for the moments and prove some of its basic properties: 

\begin{lemma}\label{lem; 53}
Let $n\in\N$ and $t\in\N\cup\{0\}$. Denote
$$f(n,t)=\sum_{k=1}^{n}(-1)^{k-1}\binom{n}{k}\frac{1}{k^{t}}.$$
Then $$f(n,t)=\begin{cases}
     1 & \textnormal{if } n = 1\textnormal{ or } t=0\\
      \sum_{s=1}^n\frac{1}{s}f(s,t-1) & \textnormal{otherwise.} \\
\end{cases}$$ In particular, $$f(n,1)=\sum_{s=1}^n\frac{1}{s} \;\;\;\;\textnormal{ and } \;\;\;\;f(n,2) = \frac{1}{2}\left(\left(\sum_{s=1}^{n}\frac{1}{s}\right)^{2}+\sum_{s=1}^{n}\frac{1}{s^{2}}\right).$$
\end{lemma}

\begin{proof}
If $n=1$ or $t=0$ the claim is clear. Suppose $n>1$ and $t>0$. Then
\begin{align}
    f(n,t)=&\sum_{k=1}^{n}(-1)^{k-1}\binom{n}{k}\frac{1}{k^{t}}\nonumber\\=&\sum_{k=1}^{n}(-1)^{k-1}\left(\binom{n-1}{k}+\binom{n-1}{k-1}\right)\frac{1}{k^{t}}\nonumber\\=&f(n-1,t)+\frac{1}{n}f(n,t-1)\nonumber
\end{align} and the claim follows by induction.

Now, the formula for $t=1$ is well known (e.g. \cite[Example 3.7]{andreescu2003path}) but our proof seems to be somewhat simpler and coincides with \cite[Example 3 on p. 4]{riordan1968combinatorial}: $$f(n,1)=\sum_{s=1}^n\frac{1}{s}f(s,0)=\sum_{s=1}^n\frac{1}{s}.$$
Finally, for $t=2$:
$$f(n,2) =  \sum_{s=1}^n\frac{1}{s}f(s,1)=\sum_{s=1}^n\frac{1}{s}\sum_{k=1}^s\frac{1}{k}= \frac{1}{2}\left(\left(\sum_{s=1}^{n}\frac{1}{s}\right)^{2}+\sum_{s=1}^{n}\frac{1}{s^{2}}\right)$$ where the last equality is due to \cite[(6.71)]{graham1989concrete} (or \cite[(3.62)]{alzer2006series}).
\end{proof}

\begin{remark}\label{rem; 33}
\begin{enumerate}
    \item In the literature, the numbers $\sum_{s=1}^n\frac{1}{s}$ are called \emph{harmonic numbers} and are denoted by $H_n, \;n\in\N$ (e.g. \cite[6.54]{graham1989concrete}) and for $t>1$ the partial sums of Riemann's zeta  function $\zeta(t)$ are called \emph{harmonic numbers of order $t$} (e.g. \cite[(6.61)]{graham1989concrete}), or \emph{generalized harmonic numbers} (e.g. \cite[1.1]{choi2011some}),  i.e. $$H_n^{(t)} = \sum_{s=1}^n\frac{1}{s^t}, \;n\in\N.$$ Thus, we may write $$f(n,1)=H_n\;\;\;\;\;\text{ and }\;\;\;\;\; f(n,2)=\frac{1}{2}\left(H_n^2+H_n^{(2)}\right).$$
    \item A natural question that arises in light of the previous remark is whether for $t>2$ it is possible to express $f(n,t)$ in terms of $H_n$ and $H_n^{(s)}, 2\leq s \leq t$. 
\end{enumerate}
\end{remark}

The following lemma enables to  simplify significantly the formula for the moments:

\begin{lemma}\label{lem; LET}
Let $t\in\N$ and let $i\in\{0,1,\ldots, t\}$. Then
\begin{equation}\label{eq; bin}\sum_{l=i+1}^{t+1}(-1)^{t+l-1}\binom{n-1+t}{n+l-2}\binom{n+l-2}{n+i-1}=\begin{cases}0&\textnormal{ if } i<t\\1 & \textnormal{ if } i=t.\end{cases}\end{equation}
\end{lemma}
\begin{proof}
Setting $m=i+1, s=t+1$ and $p=n-2$ it is obvious that (\ref{eq; bin}) is equivalent to  \begin{equation}\label{eq; bin1}\sum_{l=m}^{s}(-1)^{s+l}\binom{p+s}{p+l}\binom{p+l}{p+m}=\begin{cases}0&\textnormal{ if } m<s\\1 & \textnormal{ if } m=s\end{cases}\end{equation}
where $s,p\in\N\cup\{0\}$ and $m\in\{0,1,\ldots,s\}$. Since $$\binom{p+s}{p+l}\binom{p+l}{p+m}=\frac{(s+1)\cdots(s+p)}{(m+1)\cdots(m+p)}\binom{s}{l}\binom{l}{m},$$ it suffices to establish (\ref{eq; bin1}) for $p=0$ but this is well known (e.g. \cite[Example 2]{riordan1968combinatorial}).
\end{proof}

As a last preparation, we define a function $F(x,y,n,r)$ and show how it is related to $f(n,i)$. For our purpose only the case $x=1,y=2$ is relevant but in the proof it is advantageous to consider the more general case. Again, a special case of our result is known (e.g. \cite[(6.76)]{graham1989concrete} where it was stated while solving \cite{UNGAR}).

\begin{lemma}
Let $n\in\N, x,y\in\R$ and $r\in\N\cup\{0\}$. Denote $$F(x,y,n,r)=\sum_{k=0}^{n-1}(-1)^{k}\binom{n}{k}\frac{((n-k)x-y)^{n-1+r}}{(n-k)^r}.$$ Then \begin{equation}\label{eq; 14}F(x,y,n,r)=\sum_{i=0}^{r}x^{r-i}y^{n-1+i}\binom{n+r-1}{n+i-1}(-1)^{i}f(n,i).\end{equation} \iffalse
In particular, \begin{equation}\label{eq; 331} F(1,2,n,r)=2^{n-1}\begin{cases}
     1& \textnormal{ if } r=0\\
    n-2\sum_{k=1}^{n}\frac{1}{k} & \textnormal{ if } r=1\\
    \frac{n(n+1)}{2}-2(n+1)\sum_{k=1}^{n}\frac{1}{k}+2\left(\left(\sum_{k=1}^{n}\frac{1}{k}\right)^{2}+\sum_{k=1}^{n}\frac{1}{k^{2}}\right) & \textnormal{ if } r=2.
\end{cases}\end{equation}\fi
\end{lemma}

\begin{proof}
We proceed by induction on $r$. Let $r=0$. Then
 \begin{align}
 F(x,y,n,0)=&\sum_{k=0}^{n-1}(-1)^{k}\binom{n}{k}((n-k)x-y)^{n-1}\nonumber\\=&\sum_{k=0}^{n-1}(-1)^{k}\binom{n}{k}\sum_{l=0}^{n-1}\binom{n-1}{l}(n-k)^lx^l(-y)^{n-1-l}\nonumber\\ =&\sum_{l=0}^{n-1}\left(\binom{n-1}{l}(-y)^{n-1-l}\sum_{k=1}^n(-1)^{n-k}\binom{n}{k}k^l\right)x^l\nonumber\\ = &(-y)^{n-1}\sum_{k=1}^n(-1)^{n-k}\binom{n}{k}=y^{n-1}\nonumber.
\end{align} where in the fourth equality we used that for $1\leq l\leq n-1$, by Euler's formula for $n$th differences of powers (e.g. \cite[(2.1)]{doi:10.1080/00029890.1978.11994613}), $$\sum_{k=1}^n(-1)^{n-k}\binom{n}{k}k^l=0.$$ Now, the right hand side of (\ref{eq; 14}) equals 
$x^0y^{n-1}\binom{n-1}{n-1}(-1)^0f(n,0)= y^{n-1}$.

Assume now that (\ref{eq; 14}) holds for $r$. We shall prove that it holds for $r+1$. Replace all occurrences of $x$ in (\ref{eq; 14}) with $t$ and integrate over $[0,x]$. We obtain $$\int_0^xF(t,y,n,r)dt=\sum_{i=0}^{r}y^{n-1+i}\binom{n+r-1}{n+i-1}(-1)^{i}f(n,i)\int_0^xt^{r-i}dt\Longrightarrow$$
$$\int_0^x\sum_{k=0}^{n-1}(-1)^{k}\binom{n}{k}\frac{((n-k)t-y)^{n-1+r}}{(n-k)^r}dt=\sum_{i=0}^{r}y^{n-1+i}\binom{n+r-1}{n+i-1}(-1)^{i}f(n,i)\int_0^xt^{r-i}dt\Longrightarrow$$
\begin{align}
\sum_{k=0}^{n-1}(-1)^{k}\binom{n}{k}\frac{((n-k)x-y)^{n+r}}{(n-k)^{r+1}(n+r)}-&\sum_{k=0}^{n-1}(-1)^{k}\binom{n}{k}\frac{(-y)^{n+r}}{(n-k)^{r+1}(n+r)}\nonumber\\=&\sum_{i=0}^{r}x^{r-i+1}y^{n-1+i}\frac{1}{r-i+1}\binom{n+r-1}{n+i-1}(-1)^{i}f(n,i)\Longrightarrow\nonumber\\F(x,y,n,r+1)=&y^{n+r}\sum_{k=0}^{n-1}(-1)^{k+n+r}\binom{n}{k}\frac{1}{(n-k)^{r+1}}\nonumber\\+&\sum_{i=0}^{r}x^{r-i+1}y^{n-1+i}\frac{n+r}{r-i+1}\binom{n+r-1}{n+i-1}(-1)^{i}f(n,i)\Longrightarrow\nonumber\\F(x,y,n,r+1)=&y^{n+r}(-1)^{r+1}\sum_{k=1}^{n}(-1)^{k-1}\binom{n}{k}\frac{1}{k^{r+1}}\nonumber\\+&\sum_{i=0}^{r}x^{r-i+1}y^{n-1+i}\binom{n+r}{n+i-1}(-1)^{i}f(n,i)\Longrightarrow\nonumber\\F(x,y,n,r+1)=&\sum_{i=0}^{r+1}x^{r-i+1}y^{n-1+i}\binom{n+r}{n+i-1}(-1)^{i}f(n,i).\nonumber\end{align}
\iffalse
\begin{align}
xy^{n-1} = & \int_0^x y^{n-1}dt\nonumber \\=& \int_0^x\sum_{k=0}^{n-1}(-1)^{k}\binom{n}{k}((n-k)t-y)^{n-1}dt \nonumber \\ = &\frac{1}{n}\sum_{k=0}^{n-1}(-1)^{k}\binom{n}{k}((n-k)x-y)^n\frac{1}{n-k}-\frac{y^n}{n}\sum_{k=0}^{n-1}(-1)^{n+k}\binom{n}{k}\frac{1}{n-k}. \nonumber
\end{align}
Setting $x=1$ we have
$$ \frac{1}{n}\sum_{k=0}^{n-1}(-1)^{k}\binom{n}{k}(n-k-y)^n\frac{1}{n-k}=y^{n-1} +\frac{y^n}{n}\sum_{k=0}^{n-1}(-1)^{n+k}\binom{n}{k}\frac{1}{n-k}.$$ 
\fi
\iffalse
Finally, 
\begin{align}
\sum_{k=0}^{n-1}(-1)^{n+k}\binom{n}{k}\frac{1}{n-k}=&\sum_{k=0}^{n-1}(-1)^{n+k}\binom{n}{k}\int_0^1 t^{n-k-1}dt\nonumber\\ = &\int_0^1\frac{\sum_{k=0}^{n}\binom{n}{k} (-t)^{n-k} -1}{t}dt\nonumber\\ = &\int_0^1\frac{(1-t)^n -1}{t}dt\nonumber\\ = &-\int_0^1\frac{u^n -1}{u-1}du\nonumber\\ = &-\int_0^1\sum_{k=0}^{n-1}u^kdu=-\sum_{k=1}^n\frac{1}{k}.\nonumber
\end{align}
\fi
\end{proof}

Two methods for calculating the expectation are given in \cite{onn2011generating}. Our proof is similar to their geometric proof in that $\Delta^{n,u}$ is divided into portions determined by the index of the maximal coordinate. 

\begin{theorem}\label{thm; 5.7}
Let $\Theta\sim\mathcal{U}(\Delta^{n,u})$ and let $t\in\N$. Then $$\E{\max(\Theta)^t}=\frac{u^t}{\binom{n-1+t}{t}}f(n,t).$$ In particular, $$\E{\max(\Theta)} = \frac{u}{n}\sum_{k=1}^{n}\frac{1}{k}\;\;\;\textnormal{ and }\;\;\;\textnormal{Var}(\max(\Theta))=\frac{u^{2}}{n^{2}(n+1)}\left(n\sum_{k=1}^{n}\frac{1}{k^{2}}-\left(\sum_{k=1}^{n}\frac{1}{k}\right)^{2}\right).$$
\end{theorem}

\begin{proof}
For $1\leq i\leq n$ denote $$A_i = \{(\theta_1,\ldots,\theta_n)\in \Delta^{n,u}\;|\;\theta_i = \max\{\theta_1,\ldots,\theta_n\}\}.$$ Since the intersection of any two of the $A_i$'s is a zero set in $\R^{n-1}$, by the inclusion-exclusion principle, we have
\begin{align}\E{\max(\Theta)^t} = &  \frac{1}{\textnormal{Vol}(\Delta^{n,u})}\int_{\Delta^{n,u}}\max_{1\leq i\leq n}\{\theta_i^t\}dV\nonumber\\ = & \frac{1}{\textnormal{Vol}(\Delta^{n,u})}\sum_{i=1}^n\int_{A_i}\theta_i^tdV\nonumber\\=&\frac{n}{\textnormal{Vol}(\Delta^{n,u})}\int_{A_1}\theta_1^tdV\nonumber\\=&\frac{n(n-1)}{u^{n-1}}\int_{u/n}^u\theta_1^t  (u-\theta_1)^{n-2}\PP(\max(\Theta^{n-1,u-\theta_1})\leq\theta_1)d\theta_1\label{eq; 132}
\end{align}
where the third equality is due to symmetry. Notice that for $0\leq m\leq n-3$ it holds $$\frac{u}{n-m}\leq\theta_{1}\leq\frac{u}{n-m-1}\iff \frac{u-\theta_{1}}{n-1-m}\leq\theta_{1}\leq\frac{u-\theta_{1}}{n-1-m-1}$$ and for $m=n-2$: $$\frac{u}{2}\leq\theta_1\leq u\Longrightarrow \theta_1\geq u-\theta_1.$$ It follows that 
\begin{align}
    (\ref{eq; 132}) = & \frac{n(n-1)}{u^{n-1}}\Bigg(
    \sum_{m=0}^{n-3}\int_{\frac{u}{n-m}}^{\frac{u}{n-m-1}}\theta_1^t  \sum_{k=0}^{m}\binom{n-1}{k}(-1)^k((n-1-k)\theta_1-u+\theta_1)^{n-2}d\theta_1 + \nonumber \\ &\hspace{200pt}\int_{\frac{u}{2}}^{u}\theta_1^t(u-\theta_1)^{n-2}  d\theta_1\Bigg)\nonumber\\= &\frac{n(n-1)}{u^{n-1}}\Bigg(
    \sum_{k=0}^{n-3}\sum_{m=k}^{n-3}\binom{n-1}{k}(-1)^k\int_{\frac{u}{n-m}}^{\frac{u}{n-m-1}}\theta_1^t  ((n-k)\theta_1-u)^{n-2}d\theta_1 + \int_{\frac{u}{2}}^{u}\theta_1^t(u-\theta_1)^{n-2}  d\theta_1\Bigg)\nonumber\\= &\frac{n(n-1)}{u^{n-1}}\Bigg(
    \sum_{k=0}^{n-3}\binom{n-1}{k}(-1)^k\int_{\frac{u}{n-k}}^{\frac{u}{2}}\theta_1^t ((n-k)\theta_1-u)^{n-2}d\theta_1 +\int_{\frac{u}{2}}^{u}\theta_1^t(u-\theta_1)^{n-2}  d\theta_1\Bigg)\nonumber
    \\=&\frac{n(n-1)}{u^{n-1}}\sum_{k=0}^{n-1}\binom{n-1}{k}(-1)^k\int_{\frac{u}{n-k}}^{\frac{u}{2}}\theta_1^t  ((n-k)\theta_1-u)^{n-2}d\theta_1 \nonumber\\=&\frac{n(n-1)}{u^{n-1}}\sum_{k=0}^{n-1}\binom{n-1}{k}(-1)^k\frac{u^{t+n-1}}{2^{t+n-1}}\sum_{l=1}^{t+1}(-1)^{l-1}\frac{(n-k-2)^{n+l-2}\prod_{j=1}^{l-1}(t-j+1)}{(n-k)^{l}\prod_{j=1}^{l}(n+j-2)} \nonumber\\=&\frac{u^t}{2^{t+n-1}}\sum_{l=1}^{t+1}(-1)^{l-1}\frac{\prod_{j=1}^{l-1}(t-j+1)}{\prod_{j=2}^{l}(n+j-2)}\sum_{k=0}^{n-1}\binom{n}{k}(-1)^k\frac{(n-k-2)^{n+l-2}}{(n-k)^{l-1}}\nonumber\\=&\frac{u^t}{2^{t+n-1}}\sum_{l=1}^{t+1}(-1)^{l-1}\frac{\prod_{j=1}^{l-1}(t-j+1)}{\prod_{j=1}^{l-1}(n+j-1)}F(1,2,n,l-1)\nonumber\\=&\frac{u^t}{2^{t+n-1}}\sum_{l=1}^{t+1}(-1)^{l-1}\frac{\prod_{j=1}^{l-1}(t-j+1)}{\prod_{j=1}^{l-1}(n+j-1)}\sum_{i=0}^{l-1}2^{n-1+i}\binom{n+l-2}{n+i-1}(-1)^{i}f(n,i)\nonumber\\=&\frac{u^{t}}{2^{t}}\binom{n-1+t}{t}^{-1}\sum_{i=0}^{t}(-1)^{i}2^{i}f(n,i)\sum_{l=i+1}^{t+1}(-1)^{l-1}\binom{n-1+t}{n+l-2}\binom{n+l-2}{n+i-1}\nonumber\\=&\frac{u^t}{\binom{n-1+t}{t}}f(n,t).\nonumber
\end{align} Taking $t=1,2$, it follows from Lemma \ref{lem; 53} that $$\E{\max(\Theta)}=\frac{u}{n}\sum_{k=1}^{n}\frac{1}{k}$$ and $$\E{\max(\Theta)^2}=\frac{u^{2}}{n(n+1)}\left(\left(\sum_{k=1}^{n}\frac{1}{k}\right)^{2}+\sum_{k=1}^{n}\frac{1}{k^{2}}\right).$$ Thus, $$\textnormal{Var}(\max(\Theta))=\frac{u^{2}}{n^{2}(n+1)}\left(n\sum_{k=1}^{n}\frac{1}{k^{2}}-\left(\sum_{k=1}^{n}\frac{1}{k}\right)^{2}\right).$$
\end{proof}

\section{Randomness testing}\label{sec; app}

In this section we propose a possible application of the restrictiveness results that is inspired by the works of \cite{onn2011generating} and \cite{weissman2011testing}.

Suppose we wish to test the validity of a certain random number generator that generates numbers in $[0,1]$. That is, we want to know whether the numbers in a given sequence of real numbers are iid as $\mathcal{U}([0,1])$ (the null hypothesis). It is well known (e.g. \cite[p. 71]{gentle2003random}) that there cannot be a uniformly most powerful test and that different tests detect different deviations from the null hypothesis (\cite[pp. 4-5]{l2007testu01}). The following procedure was suggested by \cite{onn2011generating} as an additional test: Let $n,N\in\N$. Given a sequence of $nN$ real numbers in $[0,1]$, partition the numbers into $N$ groups, each consisting of $n$ numbers, and obtain $x_1,\ldots,x_N\in\Delta^{n+1}$ (cf. \cite[Algorithm 2.5.3]{rubinstein2016simulation}). Now compare the sample mean of $\max(x_1),\ldots,\max(x_N)$ with $\frac{1}{n+1}\sum_{k=1}^{n+1}\frac{1}{k}$ (cf. (\ref{eq; ex})). The authors noted  that it \say{seems that this diagnostic tool is quite sensitive to deviations from the iid uniform model} (\cite[p. 4]{onn2011generating}) and demonstrated this on a sequence of numbers that was generated by an autoregressive model. This model is defined as follows: Let $\alpha\in(0,1)$ and let $\left(U_i\right)_{i\in\N}$ be a sequence of iid random variables distributed as $\mathcal{U}([0,1])$. Construct a sequence of random variables $\left(X_i\right)_{i\in\N}$ by setting $X_1\sim\mathcal{U}([0,1])$ and for $i\in\N$ define $$X_{i+1} = \alpha X_i+ (1-\alpha)U_i.$$ %This test has been further pursued and analysed in \cite{weissman2011testing}. 

Clearly, the procedure above is not limited to the maximal coordinate. Let $\preceq$ be any stochastic order on $\Delta^n$ and denote $$p_0=\PP(\Theta\preceq\Theta')$$ where $\Theta,\Theta'\sim\mathcal{U}(\Delta^{n})$. Thus, $\preceq$ gives rise to a random variable $C\sim\text{Bernoulli}(p_0)$ such that for $x,x'\in\Delta^n$ it holds $$C(x,x')=\begin{cases} 1 &\text{if } x\preceq x'\\ 0&\text{otherwise}. \end{cases}$$ 

We propose the following procedure: For $n,N\in\N$ obtain $2nN$ real numbers in $[0,1]$ from a random numbers generator whose performance is to be evaluated. Construct a sequence $x_1,\ldots,x_{2N}\in\Delta^{n+1}$ and pair its elements to obtain a sequence $(x_1, x_2),\ldots, (x_{2N-1}, x_{2N})$. On each of the pairs $(x_{2i-1},x_{2i}), 1\leq i\leq N$ apply the function that returns $1$ if $x_{2i-1}\preceq x_{2i}$ and $0$ otherwise. Finally, perform a binomial test (e.g. \cite[24.5]{zar1999biostatistical}) on the binary sequence obtained in the previous step.

We have experimented with this procedure as follows: For each stochastic order $\preceq\in\{\leq\st, \leq\hr, \leq\lr\}$ we applied the procedure once on the iid uniform model and once on the autoregressive model with $\alpha=0.1$ (this choice of $\alpha$ was made by \cite{onn2011generating}).  We took $n=2, N=10000$ and repeated each experiment $100$ times. The code was written in Python and we used numpy.random's rand method which samples uniformly from $[0,1)$. For the $p$-value we used scipy.stats's binom\_test method. The average $p$-value together with the standard deviation are collected in Table \ref{t; 1g}. 

While the behaviour under the iid uniform model is consistent and understandable, under the autoregressive model the distinguishing abilities of the different stochastic orders differ significantly from each other. From the three stochastic orders that we examined (which are the only ones for which the corresponding Bernoulli parameter is known) only the likelihood ratio order appears to be adequate in this particular setting. We leave further experimentation and research in this respect for future work.

\begin{table}[!htbp]
     \centering
     \caption{Average $p$-value\label{t; 1g}}
     \begin{tabular}{|c||ccc|}
     \hline
     & \multicolumn{3}{c|}{\textbf{Stochastic order}} \\
        \hline \hline 
        \textbf{Model} & \textbf{st}  & \textbf{hr} & \textbf{lr} \\ \hline
        \textbf{Autoregressive} & $0.080 \pm 0.165$  & $0.524\pm 0.297$ & $0.001\pm 0.003$ \\ \hline
 $\mathbf{\mathcal{U}([0,1])}$ &$0.488\pm 0.301$  & $0.499\pm 0.270$ & $0.494\pm 0.297$ \\ \hline
 \textbf{Bernoulli parameter} &$1/3$& $1/4$    & $1/6$ \\ 
 \hline
     \end{tabular}
     \vspace{1ex}
\end{table} 

\section{Discussion}\label{sec; dis} 

In this work we addressed two aspects of random vectors $\Theta\sim\mathcal{U}(\Delta^{n})$: First, we calculated the restrictiveness of the hazard rate order and proposed a possible application of the restrictiveness results in randomness testing. Further research in this respect could be in determining the restrictiveness of additional stochastic orders, e.g. the mean residual life order (cf. \cite[2.A]{shaked2007stochastic}) for which we are still not able  to derive the analogue of Lemma \ref{lem 3}. Regarding the  application of the restrictiveness results in randomness testing, in this work we have only pointed out the possibility. Further research should determine the advantages and limitations of this method.  Second, we derived a formula for the moments of $\max(\Theta)$. The formula involves the function $f(n,t)$ defined in Lemma \ref{lem; 53}. As already mentioned in Remark \ref{rem; 33}, it would be interesting to see whether it is possible to express $f(n,t)$ in terms of the harmonic numbers and the generalized harmonic numbers.

\iffalse
\section{Acknowledgement}

We would like to thank the anonymous referees for their valuable suggestions that helped us improve this work significantly.
\fi
%\subsection{Disclosure statement}

%No potential conflict of interest was reported by the authors.

\bibliography{bibliography}

\begin{thebibliography}{10}

\bibitem{abramowitz+stegun}
M.~Abramowitz and I.~A. Stegun.
\newblock {\em Handbook of mathematical functions with formulas, graphs, and
  mathematical mables}.
\newblock Dover, New York, 1964.

\bibitem{alzer2006series}
H.~Alzer, D.~Karayannakis, and H.~M. Srivastava.
\newblock Series representations for some mathematical constants.
\newblock {\em Journal of Mathematical Analysis and Applications},
  320(1):145--162, 2006.

\bibitem{andreescu2003path}
T.~Andreescu and Z.~Feng.
\newblock {\em A path to combinatorics for undergraduates: {C}ounting
  strategies}.
\newblock Springer Science \& Business Media, 2003.

\bibitem{boland1994applications}
P.~J. Boland, E.~El-Neweihi, and F.~Proschan.
\newblock Applications of the hazard rate ordering in reliability and order
  statistics.
\newblock {\em Journal of Applied Probability}, pages 180--192, 1994.

\bibitem{choi2011some}
J.~Choi and H.~M. Srivastava.
\newblock Some summation formulas involving harmonic numbers and generalized
  harmonic numbers.
\newblock {\em Mathematical and computer Modelling}, 54(9-10):2220--2234, 2011.

\bibitem{darling1953class}
D.~A. Darling.
\newblock On a class of problems related to the random division of an interval.
\newblock {\em The Annals of Mathematical Statistics}, pages 239--253, 1953.

\bibitem{devroye2006nonuniform}
L.~Devroye.
\newblock Non-uniform random variate generation.
\newblock {\em Handbooks in operations research and management science},
  13:83--121, 2006.

\bibitem{fisher1929tests}
R.~A. Fisher.
\newblock Tests of significance in harmonic analysis.
\newblock {\em Proceedings of the Royal Society of London. Series A, Containing
  Papers of a Mathematical and Physical Character}, 125(796):54--59, 1929.

\bibitem{Fried2021OnTR}
S.~Fried.
\newblock On the restrictiveness of the usual stochastic order and the
  likelihood ratio order.
\newblock {\em Statistics and Probability Letters}, 170:109012, 2021.

\bibitem{garwood1940application}
F.~Garwood.
\newblock An application of the theory of probability to the operation of
  vehicular-controlled traffic signals.
\newblock {\em Supplement to the Journal of the Royal Statistical Society},
  7(1):65--77, 1940.

\bibitem{gentle2003random}
J.~E. Gentle.
\newblock {\em Random number generation and Monte Carlo methods}, volume 381.
\newblock Springer, 2003.

\bibitem{doi:10.1080/00029890.1978.11994613}
H.~W. Gould.
\newblock Euler's formula for $n$th differences of powers.
\newblock {\em The American Mathematical Monthly}, 85(6):450--467, 1978.

\bibitem{graham1989concrete}
R.~L. Graham, D.~E. Knuth, and O.~Patashnik.
\newblock {\em Concrete mathematics: A foundation for computer science}.
\newblock Addison-Wesley, Reading, 1989.

\bibitem{l2007testu01}
P.~L'{E}cuyer and R.~Simard.
\newblock Test{U}01: A {C} library for empirical testing of random number
  generators.
\newblock {\em ACM Transactions on Mathematical Software (TOMS)}, 33(4):1--40,
  2007.

\bibitem{marichal2008slices}
J.-L. Marichal and M.~J. Mossinghoff.
\newblock Slices, slabs, and sections of the unit hypercube.
\newblock {\em Online Journal of Analytic Combinatorics}, 3:1--11, 2008.

\bibitem{munkres2018analysis}
J.~R. Munkres.
\newblock {\em Analysis on manifolds}.
\newblock CRC Press, 2018.

\bibitem{onn2011generating}
S.~Onn and I.~Weissman.
\newblock Generating uniform random vectors over a simplex with implications to
  the volume of a certain polytope and to multivariate extremes.
\newblock {\em Annals of Operations Research}, 189(1):331--342, 2011.

\bibitem{pinelis2019order}
I.~Pinelis.
\newblock Order statistics on the spacings between order statistics for the
  uniform distribution.
\newblock {\em arXiv preprint \url{arXiv:1909.06406}}, 2019.

\bibitem{riordan1968combinatorial}
J.~Riordan.
\newblock {\em Combinatorial identities}.
\newblock Wiley series in probability and mathematical statistics. Wiley, 1968.

\bibitem{rubinstein2016simulation}
R.~Y. Rubinstein and D.~P. Kroese.
\newblock {\em Simulation and the Monte Carlo method}, volume~10.
\newblock John Wiley \& Sons, 2016.

\bibitem{shaked2007stochastic}
M.~Shaked and J.~G. Shanthikumar.
\newblock {\em Stochastic orders}.
\newblock Springer Science \& Business Media, 2007.

\bibitem{shurman2016calculus}
J.~M. Shurman.
\newblock {\em Calculus and analysis in Euclidean space}.
\newblock Springer, 2016.

\bibitem{UNGAR}
P.~Ungar.
\newblock Problem {E}3052: A sum involving {S}tirling numbers.
\newblock {\em American Mathematical Monthly}, 94:185--186, 1987.

\bibitem{weissman2011testing}
I.~Weissman.
\newblock Testing for serial correlation by means of extreme values.
\newblock {\em Reliability: Theory \& Applications}, 6(4 (23)), 2011.

\bibitem{whitworth1897dcc}
W.~A. Whitworth.
\newblock {\em DCC Exercises: Including hints for the solution of all the
  questions in choice and chance}.
\newblock D. Bell, 1897.

\bibitem{zar1999biostatistical}
Jerrold~H Zar.
\newblock {\em Biostatistical analysis}.
\newblock Pearson Education India, 1999.

\end{thebibliography}
\bibliographystyle{plain}
\end{document}